\documentclass[10pt,reqno]{amsart}
\usepackage[utf8]{inputenc}
\usepackage{amsmath}
\usepackage{amsfonts}
\usepackage{amssymb}
\usepackage{graphicx}
\usepackage[left=2.4cm,right=2.4cm,top=2.5cm,bottom=2.5cm]{geometry}
\numberwithin{equation}{section}
\graphicspath{ {images/} }
\usepackage{amsmath,amssymb, amsthm} 
\usepackage{lipsum}
\usepackage{stmaryrd}
\usepackage{dsfont}
\usepackage{comment}
\usepackage[all]{xy}

\makeatletter
\def\@settitle{\begin{center}%
  \baselineskip14\p@\relax
  \bfseries
  \uppercasenonmath\@title
  \@title
  \ifx\@subtitle\@empty\else
     \\[1ex]\uppercasenonmath\@subtitle
     \footnotesize\mdseries\@subtitle
  \fi
  \end{center}%
}
\def\subtitle#1{\gdef\@subtitle{#1}}
\def\@subtitle{}
\makeatother

\usepackage{tikz-cd}
\usepackage{hyperref}
\usepackage{bm}
\hypersetup{colorlinks, linkcolor=blue, citecolor=black, urlcolor=black}

\usepackage[english]{babel}
\usepackage[colorinlistoftodos]{todonotes}

\theoremstyle{plain}
\newtheorem{thm}{Theorem}[subsection] % reset theorem numbering for each chapter

\usepackage{mathrsfs}

\theoremstyle{definition}

\newtheorem{hyp}[thm]{Hypothesis}

\newtheorem{rem}[thm]{Remark}
\newtheorem*{rmk-intro}{Remark}
\theoremstyle{definition}

\theoremstyle{plain}
\newtheorem{prop}[thm]{Proposition}
\theoremstyle{plain}
\newtheorem{lemma}[thm]{Lemma}
\theoremstyle{plain}
\newtheorem{cor}[thm]{Corollary}
\theoremstyle{plain}
\newtheorem{conj}[thm]{Conjecture}
\theoremstyle{plain}

\theoremstyle{plain}


\newcounter{parentnumber}

\DeclareMathOperator{\ord}{ord}

\newcommand{\Gal}{\operatorname{Gal}}

\usepackage{bbm}

\makeatletter  
\newcommand{\colim@}[2]{%
  \vtop{\m@th\ialign{##\cr
    \hfil$#1\operator@font colim$\hfil\cr
    \noalign{\nointerlineskip\kern1.5\ex@}#2\cr
    \noalign{\nointerlineskip\kern-\ex@}\cr}}%
}
\newcommand{\colim}{%
  \mathop{\mathpalette\colim@{\rightarrowfill@\scriptscriptstyle}}\nmlimits@
}
\renewcommand{\varprojlim}{%
  \mathop{\mathpalette\varlim@{\leftarrowfill@\scriptscriptstyle}}\nmlimits@
}
\renewcommand{\varinjlim}{%
  \mathop{\mathpalette\varlim@{\rightarrowfill@\scriptscriptstyle}}\nmlimits@
}
\makeatother

\usepackage{multicol}

\newcommand{\Z}{\mathbb{Z}}
\newcommand{\Q}{\mathbb{Q}}

\newcommand{\X}{\mathfrak{X}}

\font\wncyr=wncyr9.8
\newcommand{\sha}{\text{\wncyr{W}}}

\newcommand{\CM}{\mathcal{M}}

\newcommand{\Lcal}{\mathcal{L}}
\newcommand{\CF}{\mathcal{F}}
\newcommand{\Fcal}{\mathcal{F}}

\newcommand{\rH}{\mathrm{H}}

 \newcommand{\cO}{{\mathcal{O}}}
 \newcommand{\CL}{{\mathcal{L}}}
 \newcommand{\ra}{\rightarrow}
 \newcommand{\fn}{{\mathfrak{n}}}
 \newcommand{\BI}{{\mathbb {I}}}
\begin{document}
\title{Base change and Iwasawa Main Conjectures for ${\rm GL}_2$}

\author{Ashay Burungale}
\address[A.~Burungale]{The University of Texas at Austin, 2515 Speedway, Austin, TX 78712, USA}
\email{ashayk@utexas.edu }

\author{Francesc Castella}
\address[F.~Castella]{University of California Santa Barbara, South Hall, Santa Barbara, CA 93106, USA}
\email{castella@ucsb.edu}

\author{Christopher Skinner} 
\address[C.~Skinner]{Princeton University, Fine Hall, Washington Road, Princeton, NJ 08544-1000, USA}
\email{cmcls@princeton.edu}

\begin{abstract}
Let $E$ be an elliptic curve defined over $\Q$ of conductor $N$, $p$ an odd prime of good ordinary reduction such that $E[p]$ is an irreducible Galois module, and $K$ an imaginary quadratic field with all primes dividing $Np$ split. 
We prove Iwasawa Main Conjectures for the $\Z_p$-cyclotomic and $\Z_p$-anticyclotomic deformations of 
$E$ over $\Q$ and $K$ respectively, dispensing with any of the ramification hypotheses on $E[p]$ in previous works. 

The strategy employs base change  and the two-variable zeta element associated to $E$ over $K$, via which the sought after main conjectures are deduced from Wan's divisibility towards a three-variable 
main conjecture for $E$ over a quartic CM field containing $K$ and certain Euler system divisibilities.

As an application, we prove cases of the two-variable main conjecture for $E$ over $K$. 
The aforementioned one-variable main conjectures imply the $p$-part of the conjectural Birch and Swinnerton-Dyer formula for $E$ if $\ord_{s=1}L(E,s)\leq 1$. They are also an ingredient in the proof of Kolyvagin's conjecture and its cyclotomic variant in our joint work with Grossi \cite{BCGS}. 

\end{abstract}

%\subjclass[2020]{11G05, 11G40 (Primary); 11R23 (Secondary)}
\date{\today}
\maketitle
\tableofcontents

%\addtocontents{toc}{\protect\setcounter{tocdepth}{1}}

\section{Introduction}\label{subsec:app-results}

Let $E$ be an elliptic curve defined over $\Q$, $p$ an odd prime of good ordinary reduction for $E$, and $K$ an imaginary quadratic  field. In this paper we study Iwasawa theory of $E$ along the cyclotomic $\Z_p$-extension of $\Q$ and the anticyclotomic $\Z_p$-extension of $K$, proving corresponding Iwasawa Main Conjectures (see~Theorems~\ref{thm:cyc},~\ref{thm:HPMC}~and~\ref{thm:BDP}).
% in situations %(assuming $E[p]$ is an irreducible $G_\Q$-module)
%where previous methods required the $p$-indivisibility of some Tamagawa number of $E$.

The main text also treats the case of weight two elliptic newforms with good ordinary reduction at $p$. In the introduction, for simplicity, we present results for elliptic curves. 

%\subsection{Main results}

\subsection{Cyclotomic Main Conjecture}

Let $\Q_\infty$ be the cyclotomic $\Z_p$-extension of $\Q$. Put $\Gamma={\rm Gal}(\Q_\infty/\Q)$, and let $\Lambda=\Z_p[\![\Gamma]\!]$ be the cyclotomic Iwasawa algebra. 
%Following Mazur \cite{mazur-towers}, 

We consider the classical Selmer group ${\rm Sel}_{p^\infty}(E/\Q_\infty)=\varinjlim_n{\rm Sel}_{p^\infty}(E/\Q_n)$, where $\Q_n$ is the subfield of $\Q_\infty$ with $[\Q_n:\Q]=p^n$. Its Pontryagin dual 
\[
\mathfrak{X}_{\rm ord}(E/\Q_\infty):=
%{\rm Sel}_{p^\infty}(E/\Q_\infty)^\vee=
{\rm Hom}_{\Z_p}({\rm Sel}_{p^\infty}(E/\Q_\infty),\Q_p/\Z_p)
\]
is a finitely generated $\Lambda$-module. Let $\CL_p^{}(E/\Q)\in\Lambda\otimes\Q_p$ be the $p$-adic $L$-function attached to $E$ by Mazur--Swinnerton-Dyer \cite{M-SwD}. In \cite{mazur-towers}, Mazur conjectured the following.  

\begin{conj}[Mazur's Main Conjecture]
\label{conj:cyc} 
The $\Lambda$-module $\X_{\rm ord}(E/\Q_{\infty})$ is torsion, with
\[
{\rm ch}_\Lambda\left( \X_{\rm ord}(E/\Q_{\infty})\right)=(\Lcal_p(E/\Q))
\]
as ideals in $\Lambda$. 
\end{conj}

Note that implicit in Conjecture~\ref{conj:cyc} is the integrality statement $\Lcal_p(E/\Q)\in\Lambda$; this is most well-understood  under the assumption that $p$ is odd and
\begin{equation}\label{eq:irr}
\tag{irr$_\Q$}
\text{$E[p]$ is an irreducible $G_\Q$-module}
\end{equation}
(see \cite[Prop.~3.1]{greenvats}) where $G_{\Q}:=\Gal(\overline{\Q}/\Q)$ is
the absolute Galois group of $\Q$. (We similarly use $G_L$ to denote the absolute Galois group of a number field $L$.) Let $T$ be the $p$-adic Tate module of $E$. 

In \cite{kato-euler-systems}, Kato proved the $\Lambda$-torsionness of $\mathfrak{X}_{\rm ord}(E/\Q_\infty)$ and the inclusion $p^c\cdot\Lcal_p(E/\Q)\in{\rm ch}_\Lambda(\mathfrak{X}_{\rm ord}(E/\Q_\infty))$ for some $c\geq 0$, with $c=0$ when $T$ has large image. These results are consequences of his seminal construction of an Euler system for $T$.
%thereby giving one of the divisibilities predicted by Conjecture~\ref{conj:cyc}. 
Let $N$ be the conductor of $E$. Assuming further that
\begin{equation}\label{eq:mult}\tag{mult}
\textrm{there exists a prime $q\Vert N$ such that $E[p]$ is ramified at $q$,}
\end{equation} 
the converse divisibility, and  hence Conjecture~\ref{conj:cyc} was proved by Skinner--Urban \cite{skinner-urban}. Their work employs Eisenstein congruences on the unitary group ${\rm GU}(2,2)$ over  imaginary quadratic fields. 

Our main result towards Conjecture~\ref{conj:cyc} removes the hypothesis \eqref{eq:mult}:

\begin{thm}\label{thm:cyc}
Let $E$ be an elliptic curve defined over $\Q$ and $p$ a prime of good ordinary reduction for $E$. 
\begin{itemize}
\item[(a)] If $p>3$ satisfies \eqref{eq:irr}, then $\X_{\rm ord}(E/\Q_{\infty})$ is $\Lambda$-torsion, with
\[
{\rm ch}_{\Lambda_\Q}\left(\X_{\rm ord}(E/\Q_{\infty})\right)=\bigl(\Lcal_p(E/\Q)\bigr)
\]
in $\Lambda\otimes\Q_p$. 
\item[(b)] If in addition
\begin{equation}\label{im}\tag{im}
\text{there exists an element $\sigma \in G_{\Q(\mu_{p^\infty})}$ such that $T/(\sigma-1)T\simeq\Z_p$,} %is a free $\Z_p$-module of rank one,}
\end{equation}
then the equality holds in $\Lambda$, and hence Conjecture~\ref{conj:cyc} holds.
\end{itemize}
\end{thm}

\begin{rem}\label{rem:cyc}\hfill
\begin{enumerate}
\item[(i)]{} For non-CM curves the condition \eqref{im} holds for all sufficiently large primes $p$ by Serre's open image theorem \cite{serre}.  In fact, it is expected that $p\geq 37$ suffices. 
\item[(ii)]{} The only prior result towards Conjecture~\ref{conj:cyc}
%An earlier result on Conjecture~\ref{conj:cyc} 
without assuming the hypothesis \eqref{eq:mult} is due to Wan \cite{wan-hilbert}, based on Eisenstein congruences on the unitary group ${\rm GU}(2,2)$ over CM fields. 
However, it is conditional on a $p$-integral comparison of certain automorphic periods, which still remains open. 
%,was obtained by X.\,Wan \cite{wan-hilbert} from his work on Eisenstein congruences on ${\rm GU}(2,2)$ over CM fields. 
Our proof of Theorem~\ref{thm:cyc} relies on a main result of \cite{wan-hilbert} but sidesteps the  period comparison.
\item[(iii)]{} Under the condition \eqref{eq:irr}, the essential case excluded by Theorem~\ref{thm:cyc} is that of (residually) dihedral primes $p$. 
\end{enumerate}
\end{rem}

%We note that \eqref{im} is satisfied if the $G_\Q$-action of $E[p]$ has big image, and so by Serre's open image theorem 
%\cite{serre} this condition holds for sufficiently  large\footnote{It is expected that $p\geq 37$ suffices.} $p$ in the non-CM case. 

\subsection{Anticyclotomic Main Conjectures}

%Let $N$ be the conductor of $E$. 
Assume that the discriminant $D_K<0$ satisfies
\begin{equation}\label{eq:intro-disc}
\textrm{$D_K$ is odd and $D_K\neq -3$.}\tag{disc}
\end{equation}
Moreover, assume that $K$ satisfies the \emph{Heegner hypothesis}, namely
\begin{equation}\label{eq:intro-Heeg}
\textrm{every prime $\ell\vert N$ splits in $K$,}\tag{Heeg}
\end{equation} 
%where as before $N$ denotes the conductor of $E$, 
and that
\begin{equation}\label{eq:intro-spl}
\textrm{$p=v\bar{v}$ splits in $K$}\tag{spl}
\end{equation}
for $v$ the prime of $K$ above $p$ induced by an embedding $\overline{\Q}\hookrightarrow\overline{\Q}_p$, which we fix throughout. 

Let $K_\infty^-/K$ be the anticyclotomic $\Z_p$-extension,  $\Gamma_K^-={\rm Gal}(K_\infty^-/K)$, and 
$\Lambda_K^-=\Z_p[\![\Gamma_K^-]\!]$ the anticyclotomic Iwasawa algebra. In view of \eqref{eq:intro-Heeg} and the $p$-ordinarity hypothesis, the Kummer images of Heegner points of $p$-power conductor give rise to a $\Lambda_K^-$-adic class
\[
\kappa_1^{\rm Heeg}\in\rH^1_{\Fcal_\Lambda}(K,\mathbf{T}).
\]
Here $\mathbf{T}=\varprojlim_n{\rm Ind}_{K_n^-/K}(T)$, with $K_n^-$ the subfield of $K_\infty^-$ with $[K_n^-:K]=p^n$, and $\rH^1_{\Fcal_\Lambda}(K,\mathbf{T})\subset\rH^1(K,\mathbf{T})$ is the compact ordinary Selmer group\footnote{See e.g. \cite[\S{4.1}]{eisenstein} for a review of the construction, whose notations we largely follow.} interpolating the classical Selmer groups $\varprojlim_{m}{\rm Sel}_{p^m}(E/K_n^-)$ as $n$ varies.  Let $\mathfrak{X}_{\rm ord}(E/K_\infty^-)$ be the Pontryagin dual of ${\rm Sel}_{p^\infty}(E/K_\infty^-)=\varinjlim_n{\rm Sel}_{p^\infty}(E/K_n^-)$. 

The formulation of a Main Conjecture in this setting is due to Perrin-Riou \cite{perrinriou}.

\begin{conj}[Heegner point Main Conjecture]\label{conj:HPMC}
The $\Lambda_K^-$-modules $\mathfrak{X}_{\rm ord}(E/K_\infty^-)$  and $\rH^1_{\Fcal_\Lambda}(K,\mathbf{T})$ have $\Lambda_K^-$-rank one, and 
\[
{\rm ch}_{\Lambda_K^-}\bigl(\X_{\rm ord}(E/K_\infty^-)_{\rm tor}\bigr)={\rm ch}_{\Lambda_K^-}\bigl(\rH^1_{\Fcal_\Lambda}(K,\mathbf{T})/(\boldsymbol{\kappa}_1^{\rm Heeg})\bigr)^2
\]
as ideals in $\Lambda_K^-$.
\end{conj}

The first general results towards  Perrin-Riou's Heegner point Main Conjecture 
 are due to Bertolini \cite{bertolini-PhD} and Howard \cite{howard}, 
 relying on the Heegner point Kolyvagin system. 
 These works  established the rank statements in Conjecture~\ref{conj:HPMC}, and the latter proved the divisibility ``$\supseteq$'' if 
\begin{equation}\label{eq:irred}
\tag{sur}
\textrm{$\bar{\rho}_E: G_{\Q}\to{\rm Aut}_{\mathbb{F}_p}(E[p])$ is surjective.}
\end{equation}
% for the characteristic ideals. 
The first cases of the opposite divisibility, and hence of Conjecture~\ref{conj:HPMC} appeared in \cite{BCK}, which builds on Wei Zhang's resolution of Kolyvagin's conjecture \cite{zhang_ind} and assumes  the hypothesis \eqref{eq:irred} in addition to certain ramification hypotheses on $E[p]$. Via level raising and rank lowering, Wei Zhang's work and \cite{BCK} rely on the results of Skinner--Urban \cite{skinner-urban}, inheriting the hypotheses therein. 
Prior to \cite{BCK}, Wan \cite{wan-GU31} established the first cases of the Heegner point Main Conjecture under a  generalized (non-classical) Heegner hypothesis. 
%are a consequence of the main result \cite{wan-GU31} of Wan, which 
It employs Eisenstein congruences on the unitary group ${\rm GU}(3,1)$.
In addition to \eqref{eq:irred}, it requires that $N$ is square-free  
%\begin{equation}\label{eq:irred}
%\tag{sur}
%\textrm{$\bar{\rho}_E: G_{\Q}=\operatorname{Gal}(\bar{\Q}/\Q)\to{\rm Aut}_{\mathbb{F}_p}(E[p])$ is surjective.}
%\end{equation}
and a ramification hypothesis on $E[p]$.
% the first cases of the other inclusion, and hence of Conjecture~\ref{conj:HPMC}, were obtained following the main result of X.\,Wan \cite{wan-GU31} on Eisenstein congruences on ${\rm GU}(3,1)$  

Our different approach %building on \cite{wan-hilbert} instead,
%to the other inclusion 
dispenses with any of the ramification hypotheses, leading to the following result.

\begin{thm}\label{thm:HPMC}
Let $E$ be an elliptic curve defined over $\Q$ of conductor $N$, $p$ be a prime of good ordinary reduction for $E$, and $K$ an imaginary quadratic field satisfying \eqref{eq:intro-disc}, \eqref{eq:intro-Heeg}, and \eqref{eq:intro-spl}. 
\begin{itemize}
\item[(a)] If $p>3$ satisfies \eqref{eq:irr}, then both $\mathfrak{X}_{\rm ord}(E/K_\infty^-)$  and $\rH^1_{\Fcal_\Lambda}(K,\mathbf{T})$ have  $\Lambda_K^-$-rank one, and 
\[
{\rm ch}_{\Lambda_K^-}\bigl(\X_{\rm ord}(E/K_\infty^-)_{\rm tor}\bigr)={\rm ch}_{\Lambda_K^-}\bigl(\rH^1_{\Fcal_\Lambda}(K,\mathbf{T})/(\boldsymbol{\kappa}_1^{\rm Heeg})\bigr)^2
\]
in $\Lambda_K^-\otimes\Q_p$. 
\item[(b)] If further $p>3$ satisfies \eqref{eq:irred}, then the equality holds in $\Lambda_K^-$ and hence %\cite[Conj.~B]{perrinriou} holds.
Conjecture~\ref{conj:HPMC} holds.
\end{itemize}
\end{thm} 
\begin{rem}
Under the condition \eqref{eq:irr}, the only case excluded by Theorem~\ref{thm:HPMC}(b) is that of (residually) dihedral primes $p$. It will be treated in \cite{bs}.
\end{rem}

Theorem~\ref{thm:HPMC} has applications to the Birch and Swinnerton-Dyer conjecture. For example, Theorem~\ref{thm:HPMC}(a) yields a $p$-converse to the Gross--Zagier and Kolyvagin theorem: 
$$
{\rm corank}_{\Z_p}{\rm Sel}_{p^\infty}(E)=1 \implies \ord_{s=1}L(E,s)=1
$$
 (cf. \cite{pCONVskinner,wan-heegner,castellaheights,CMpconverse, BST-sv, BT22}). Note that the $p$-converse does not assume finiteness of $\sha(E)[p^\infty]$, but in fact deduces it as a consequence. 
 
In light of the $\Lambda_K^-$-adic analogue of the $p$-adic Waldspurger formula of \cite{BDP} (see \cite{cas-hsieh1}),  Conjecture~\ref{conj:HPMC} is equivalent
%\footnote{See \cite{cas-hsieh1,wan-heegner,BCK}.} 
to the prediction that the $p$-adic $L$-function $\mathcal{L}_p^{\rm BDP}(E/K)\in \Lambda_K^{-,{\rm ur}}$ constructed in \emph{op.\,cit.} generates the characteristic ideal of the anticyclotomic Selmer group  $\mathfrak{X}_{\rm Gr}(E/K_\infty^-)$ 
%obtaining from $\mathfrak{X}_{\rm ord}(E/K_\infty^-)$ by requiring 
whose classes are locally trivial (resp. unrestricted) at the primes above $\overline{v}$ (resp. $v$). Here we put 
\[
\Lambda_K^{-,{\rm ur}}=\Lambda_K^-\hat\otimes_{\Z_p}\Z_p^{\rm ur},
\] 
where $\Z_p^{\rm ur}$ denotes the completion of the ring of integers of the maximal unramified extension of $\Q_p$. Hence, Theorem~\ref{thm:HPMC} also yields the following.

\begin{thm}\label{thm:BDP}
%Let $E/\Q$ be an elliptic curve of conductor $N$ with associated newform $f\in S_2(\Gamma_0(N))$, let $p$ be a prime of good ordinary reduction for $E$, and let $K$ be an imaginary quadratic field satisfying \eqref{eq:intro-Heeg}, \eqref{eq:intro-disc}, and \eqref{eq:intro-spl}. 
Let %the notations and assumptions be the same 
$(E,p,K)$ be as in Theorem~\ref{thm:HPMC}. 
\begin{itemize}
\item[(a)] If $p>3$ satisfies \eqref{eq:irr}, then $\mathfrak{X}_{\rm Gr}(E/K_\infty^-)$ is $\Lambda_K^-$-torsion, and 
\[
{\rm ch}_{\Lambda_K^-}\bigl(\X_{\rm Gr}(E/K_\infty^-)\bigr)=\bigl(\Lcal_p^{\rm BDP}(E/K)\bigr)
\]
in $\Lambda_K^{-,{\rm ur}}\otimes\Q_p$. 
\item[(b)] If further $p>3$ satisfies   \eqref{eq:irred}, then the equality of characteristic ideals holds in $\Lambda_K^{-,{\rm ur}}$. 
\end{itemize}
\end{thm}

\subsection{Some applications}

\subsubsection{The Birch and Swinnerton-Dyer formula} 

A consequence of Theorems~\ref{thm:cyc} and~\ref{thm:BDP} is the following. 
% have the following application to the Birch and Swinnerton-Dyer conjecture. 

\begin{cor}\label{cor:BSD}
Let $E/\Q$ be a non-CM elliptic curve. Let $p>3$ be a prime of good ordinary reduction such that \eqref{eq:irr} and \eqref{im} hold. If $\ord_{s=1}L(E,s) = r \in \{0,1\}$, then 
$$
\bigg{|}\frac{L^{(r)}(E,1)}{{\rm Reg}(E)\cdot \Omega_{E}}\bigg{|}_{p}^{-1}=
\bigg{|}\#\sha(E) \prod_{\ell \nmid \infty}c_{\ell}(E)\bigg{|}_{p}^{-1}
$$
and hence the $p$-part of the conjectural BSD formula for $E$ is true.
\end{cor}

\begin{proof}
In the case $r=0$, this follows from the equality of characteristic ideals in Theorem~\ref{thm:cyc}, the interpolation property of $\mathcal{L}_p(E/\Q)$ at the trivial character, and the formula (up to a $p$-adic unit) in  the control theorem \cite[Thm.\,4.1]{greenberg-cetraro} for the value of a characteristic power series for $\mathfrak{X}_{\rm ord}(E/\Q_\infty)$ at the trivial character.

Similarly, for a suitably chosen imaginary quadratic field $K$, the result in the case $r=1$ follows from Theorem~\ref{thm:BDP}, the $p$-adic Waldspurger formula \cite{BDP} for the value of $\mathcal{L}_p^{\rm BDP}(E/K)$ at the trivial character, the anticyclotomic control theorem \cite[Thm.\,3.3.1]{jsw}, and the $r=0$ result for the $K$-quadratic twist of $E$. (See also \cite[\S{1}]{ICTS} for a more detailed review of these arguments.)
\end{proof}

%Note that by Serre's open image theorem \cite{serre}, in the non-CM case condition \eqref{im} holds for all sufficiently large $p$ (in fact, it is expected that $p\geq 37$ suffices). 
\begin{rem}
The condition \eqref{im} in Corollary~\ref{cor:BSD} excludes only finitely many primes $p$ (cf. Remark~\ref{rem:cyc}(i)). For an overview of prior results, the reader may refer to \cite{BST-sv,BSTW}. In contrast to them, the above result does not impose any condition on the conductor of $E$. 
\end{rem}

\subsubsection{Kolyvagin's conjecture} 
%We conclude this Introduction by noting that 
%When $p>3$ satisfies \eqref{eq:irr}, 
Theorems~\ref{thm:cyc} and~\ref{thm:BDP} %(i.e. the equalities of characteristic ideals after inverting $p$) 
are one of the key ingredients\footnote{A similar input when \eqref{eq:irr} is not satisfied is provided by the main results of \cite{eisenstein_cyc}.} in the proof of Kolyvagin's conjecture and its analogue for Kato's Euler system in the joint work \cite{BCGS} of the authors with Grossi. These conjectures assert a non-triviality of the associated Kolyvagin systems. 
When $p>3$ satisfies \eqref{eq:irred}, Theorems~\ref{thm:cyc} and~\ref{thm:BDP} are also used in \cite{BCGS} to prove the refinement of Kolyvagin's conjecture and its cyclotomic analog formulated by W.\,Zhang \cite{zhang-CDM} and C.-H.\,Kim \cite{kim}, respectively.

\subsection{On the two-variable Main Conjectures} 

Our approach to the above theorems also gives a proof of the two-variable Iwasawa Main Conjectures for $E/K$ under an additional hypothesis on $E[p]$. 

For the precise statement, following terminology 
%of Diamond 
in \cite{diamond-FLT}, consider the set of ``vexing primes'' for $E[p]$:
\[
\mathcal{V}:=\bigl\{\ell\equiv -1\,({\rm mod}\,p)\;\,\vert\;\,\textrm{ $\bar{\rho}_E\vert_{G_{\Q_\ell}}$ is irreducible and $\bar{\rho}_E\vert_{I_\ell}$ is reducible}\bigr\},
\]
where $I_\ell\subset G_{\Q_\ell}$ are inertia and decomposition groups at $\ell$, respectively. 

Let $K_\infty/K$ denote the $\Z_p^2$-extension of $K$, and put $\Gamma_K={\rm Gal}(K_\infty/K)$ and $\Lambda_K=\Z_p[\![\Gamma_K]\!]$. Let $\mathfrak{X}_{\rm ord}(E/K_\infty)$ be the Pontryagin dual of the Selmer group ${\rm Sel}_{p^\infty}(E/K_\infty)$, and let $\Lcal_p^{\rm PR}(E/K)\in\Lambda_K$ be the two-variable $p$-adic Rankin $L$-series constructed by Perrin-Riou \cite{PR-JLMS} (normalized as in \cite[\S{1.2}]{eisenstein_cyc}).

\begin{thm}\label{thm:2var-IMC}
Let $(E,p,K)$ be as in Theorem~\ref{thm:HPMC}. Assume that $\mathcal{V}=\emptyset$. 
\begin{itemize}
\item[(a)] If $p>3$ satisfies \eqref{eq:irr}, then  $\mathfrak{X}_{\rm ord}(E/K_\infty)$ is $\Lambda_K$-torsion, with
\[
{\rm ch}_{\Lambda_K}\bigl(\mathfrak{X}_{\rm ord}(E/K_\infty)\bigr)=\bigl(\Lcal_p^{\rm PR}(E/K)\bigr)
\]
in $\Lambda_K\otimes\Q_p$. 
\item[(b)] If further $p>3$ satisfies \eqref{eq:irred}, then the equality of characteristic ideals holds in $\Lambda_K$.
\end{itemize}
\end{thm}

\begin{rem} 
Since the global root number of $E$ over $K$ equals $-1$ (cf.~\eqref{eq:intro-Heeg}), Theorem~\ref{thm:2var-IMC} complements the results on the two-variable Iwasawa Main Conjecture in  \cite{skinner-urban}. On the other hand, \cite[Thm.~A]{cas-wan} established some cases of the two-variable Main Conjecture for semistable elliptic curves $E$ 
when $\varepsilon(E/K)=-1$. However, the latter results exclude the situation of the classical Heegner hypothesis and assumes a ramification hypothesis on $E[p]$. 
\end{rem}

\subsection{About the proofs}

The key new idea is to base change $E$ to a quartic CM field $M$ containing $K$ for which the main result of \cite{wan-hilbert} towards a three-variable Main Conjecture applies, and utilize the two-variable zeta element associated to $E$ over $K$ recently constructed in \cite{BSTW}. 

To begin, the main result of \cite{wan-hilbert} yields the divisibility 
\begin{equation}\label{eq:PR-div}
\bigl(\CL_p^{\rm PR}(E/K)\cdot\CL_p^{\rm PR}(E^F/K)\bigr)\supset
{\rm ch}_{\Lambda_K}\bigl(\mathfrak{X}_{\rm ord}(E/K_\infty)\bigr)\cdot{\rm ch}_{\Lambda_K}\bigl(\mathfrak{X}_{\rm ord}(E^F/K_\infty)\bigr)
\end{equation}
in $\Lambda_K\otimes\Q_p$, where $E^F$ is the quadratic twist of $E$ for the real subfield $F$ contained in $M$. In view of the two-variable zeta elements of \cite{BSTW} and their explicit reciprocity laws, this translates into the divisibility 
\begin{equation}\label{eq:Gr-div}
\bigl(\CL_p^{\rm Gr}(E/K)\cdot\CL_p^{\rm Gr}(E^F/K)\bigr)\supset
{\rm ch}_{\Lambda_K}\bigl(\mathfrak{X}_{\rm Gr}(E/K_\infty)\bigr)\cdot{\rm ch}_{\Lambda_K}\bigl(\mathfrak{X}_{\rm Gr}(E^F/K_\infty)\bigr)
\end{equation}
in $\Lambda_K^{\rm ur}\otimes\Q_p$, where $\mathcal{L}_p^{\rm Gr}(E^\cdot/K)\in\Lambda_K^{\rm ur}:=\Lambda_K\hat\otimes_{\Z_p}\Z_p^{\rm ur}$ is a two-variable $p$-adic Rankin $L$-series %(cf.~\cite[\S{1.4}]{eisenstein}) 
specializing to $\Lcal_p^{\rm BDP}(E/K)$ (up to a unit) under the natural projection
$\Lambda_K^{\rm ur}\rightarrow\Lambda_K^{-,{\rm ur}}$, and $\mathfrak{X}_{\rm Gr}(E/K_\infty)$ is the counterpart of $\mathfrak{X}_{\rm Gr}(E/K_\infty^-)$ over $K_\infty/K$. 

In view of the vanishing of the Iwasawa $\mu$-invariant of $\Lcal_p^{\rm BDP}(E/K)$ proved in \cite{hsieh-special,Bu-mu} following ideas in  \cite{hidamu}, the divisibilities \eqref{eq:PR-div} and \eqref{eq:Gr-div} both hold  integrally. 

The proof of Theorem~\ref{thm:cyc} then follows from \eqref{eq:PR-div} (for a suitably chosen $K$) by descending to the cyclotomic $\Z_p$-extension $K_\infty^+/K$ and appealing to Kato's work \cite{kato-euler-systems}. Similarly, the proof of Theorem~\ref{thm:BDP} (and hence of Theorem~\ref{thm:HPMC}) follows from \eqref{eq:Gr-div} by descending to 
%the anticyclotomic $\Z_p$-extension 
$K_\infty^-/K$ and appealing to the Kolyvagin system bound developed in  \cite{eisenstein,eisenstein_cyc} applied to the Heegner point Euler system. 
Without any restriction on $\mathcal{V}$, we thus arrive at the equality 
\begin{equation}\label{eq:product-eq}
\bigl(\CL_p^{\rm PR}(E/K)\cdot\CL_p^{\rm PR}(E^F/K)\bigr)=
{\rm ch}_{\Lambda_K}\bigl(\mathfrak{X}_{\rm ord}(E/K_\infty)\bigr)\cdot{\rm ch}_{\Lambda_K}\bigl(\mathfrak{X}_{\rm ord}(E^F/K_\infty)\bigr)
\end{equation}
and likewise for \eqref{eq:Gr-div}.  Then assuming $\mathcal{V}=\emptyset$ we separate
%\footnote{For the relevance of vexing primes to our argument, the reader may refer to proof of Theorem~\ref{thm:2var-IMC}  in \S\ref{ss:proofs}.} 
the two factors, concluding the proof of Theorem~\ref{thm:2var-IMC}. 

The above strategy also applies for the prime $p=3$. The only missing ingredient is the divisibility \eqref{eq:PR-div} established in~\cite{wan-hilbert}, which assumes $p>3$ only for a comparison of certain automorphic periods associated to Hilbert modular forms (cf.~\cite[Thm.~86]{wan-hilbert}).

\begin{rem}
An Euler system for $E/K$ extending\footnote{The construction of such an Euler system will complement the one in \cite{LLZ-K}, where the authors need to twist by a non-Eisenstein and $p$-distinguished Hecke character.} the construction in \cite{BSTW} would give rise to a divisibility
\begin{equation}\label{eq:zeta-div}
{\rm ch}_{\Lambda_K}\bigl(\mathfrak{X}_{\rm ord}(E/K_\infty)\bigr)\supset\bigl(\Lcal_p^{\rm PR}(E/K)\bigr),\nonumber
\end{equation}
possibly after inverting $p$. With this divisibility in hand, the proof of  Theorem~\ref{thm:2var-IMC} would follow from \eqref{eq:product-eq} without the additional hypothesis that $\mathcal{V}=\emptyset$.
A construction of this Euler system will appear in the work of the first-named author with Marco Sangiovanni Vincentelli. 
\end{rem}
%\subsection{Organisation}

%\begin{rem}
%In the main text we prove Theorems~\ref{thm:cyc},~\ref{thm:HPMC},~\ref{thm:BDP}, and~\ref{thm:2var-IMC} for any weight two elliptic newform with good ordinary reduction at $p$.
% in place of $E$.
%\end{rem}

%\begin{rem}
%\end{rem}

\subsection{Acknowledgements}
The authors thank Giada Grossi, Haruzo Hida, Marco Sangiovanni Vincentelli, Ye Tian and Xin Wan for helpful communications. They are grateful to the referees for 
their detailed and useful comments. During the preparation of this paper, A.B. was partially supported by the NSF grants DMS-2303864 and DMS-2302064; F.C. was partially supported by the NSF grants DMS-2101458 and DMS-2401321; C.S. was partially supported by the Simons Investigator Grant \#376203 from the Simons Foundation and by the NSF grant DMS-1901985. 

%Our approach to Theorem~\ref{app:thm-BDP} also yields the proof of a two-variable Main Conjecture over $K$. We refer the reader to \cite[$\S{1.2}$]{eisenstein_cyc} for the construction and exact interpolation property (following from Perrin-Riou's work \cite{PR-JLMS}) of the two-variable $p$-adic Rankin $L$-series $\Lcal_p^{\rm PR}(E/K)\in\Lambda_K$, and [\emph{op.\,cit.}, $\S{2.1}$] for the definition of the ordinary Selmer group 
%\[
%\mathfrak{X}_{\rm ord}(E/K_\infty)=\rH^1_{\Fcal_{\rm ord}}(K,T_pE\otimes\Lambda_K^\vee)^\vee
%\]
%(defined is the same manner as $\mathfrak{X}_{\rm ord}(E/\Q_\infty)$ in the body of the paper). 

%\begin{thm}\label{app:thm-K}
%Let $E/\Q$ be an elliptic curve of conductor $N$, let $p\nmid N$ be a prime of good ordinary reduction for $E$, and let $K$ be an imaginary quadratic field of discriminant $D_K$ prime to $Np$. Suppose $K$ satisfies \eqref{eq:intro-Heeg}, \eqref{eq:intro-disc}, and \eqref{eq:intro-spl}. 
%If $p>3$ is such that 
%\begin{equation}\label{irr_{K}}\tag{irr$_K$}
%\text{$E[p]$ is an irreducible $G_K$-representation}
%\end{equation}
%then $\mathfrak{X}_{\rm ord}(E/K_\infty)$ is $\Lambda_K$-torsion, and 
%\[
%{\rm ch}_{\Lambda_K}\bigl(\X_{\rm ord}(E/K_\infty)\bigr)=\bigl(\Lcal_p^{\rm PR}(E/K)\bigr)
%\]
%as ideals in $\Lambda_K$.
%\end{thm}

\section{Main Conjecture over CM fields} 

%Though our results primarily concern main conjectures over imaginary quadratic fields, the proof is based on a Main Conjecture over CM fields. 

In this section we briefly recall the formulation of the Iwasawa Main Conjecture over CM fields $M/F$ at the base of the proof of our main results. 
\subsection{Setting}
Let $F$ be a totally real field of degree $d=[F:\Q]$. 

Let $g\in S_{2}(\Gamma_{0}(\fn))$ be a Hilbert modular newform over $F$ of parallel weight $2$. Let $p$ be a prime with 
\begin{equation}\label{ur}\tag{ur}
\text{$p\nmid D_F$,}
\end{equation}
where $D_F$ denotes the discriminant of $F/\Q$. For a prime $\lambda$ of the Hecke field $F_g$ over $p$, let $\rho_{g}:G_{F} \rightarrow{\rm GL}_{2}(F_{g,\lambda})$ be the associated Galois representation and $V_{g}:=V_{g,\lambda}$ the underlying $F_{g,\lambda}$-vector space. Let $T_{g}\subset V_{g}$ be a $G_F$-stable $\cO:=\cO_{F_{g,\lambda}}$-lattice and,
\[
\bar{\rho}_{g}:G_{F}\rightarrow{\rm GL}_{2}(\bar{\mathbb{F}}_{p})
\]
%semi-simplification of 
the corresponding residual representation. Suppose that $g$ is ordinary at each prime $w$ of $F$ over $p$, which we abbreviate as $p$ being a prime of ordinary reduction for $g$.  
Let $0 \subset {\rm Fil}_{w}^{+}(V_{g}) \subset V_g$ be the associated filtration of $F_{g,\lambda}[G_{F_{w}}]$-modules and put ${\rm Fil}_{w}^{+}(T_{g})=T_g\cap{\rm Fil}_{w}^{+}(V_{g})$. 
%We shall often suppose that
%\begin{equation}\label{dist}\tag{dist$_F$}
%\text{the action of $G_{F_{w}}$ on $ {\rm Fil}_{w}^{+}(T_{g})$ and $T_{g}/ {\rm Fil}_{w}^{+}(T_{g})$ are by distinct characters modulo $\fm_\cO$ for any $w\vert p$,}
%\end{equation}
%where $\fm_\cO$ denotes the maximal ideal of $\cO$.

%\subsubsection{CM fields} 

Let $M$ be a CM quadratic extension of $F$ such that 
\begin{equation}\label{spl-I}\tag{spl$_{F}$}
\text{any prime of $F$ above $p$ splits in $M$.}
\end{equation}
%We often suppose
%\begin{equation}\label{exc-I}\tag{$\Delta$}
%\text{$M$ is not contained in $H_{F}$, and any prime ramified in $F/\Q$ splits in $M$,}
%\end{equation}
%where $H_F$ denotes the narrow Hilbert class field of $F$. For a place $w$ of $M$, let 
%$I_{w} \subset G_{w} \subset G_{M}$ be the inertia and decomposition subgroups. 
%
%For an ideal $\fa \subset \cO_F$, write $\fa=\fa^{+}\fa^{-}$, where $\fa^{+}$ and $\fa^-$ are precisely divisible by the primes which are split and non-split in the extension $M/F$ respectively.
%
Denote by $\Gamma_{M}^-$ (resp.~$\Gamma_{M}^+$) the Galois group of 
the anticyclotomic $\Z_p^{d}$-extension $M_\infty/M$ (resp. cyclotomic $\Z_p$-extension $M_\infty^+/M$). Let $M_{\infty}=M_{\infty}^-M_{\infty}^+$ be the compositum, and put $\Gamma_{M}=\Gal(M_{\infty}/M)$ and $\Lambda_{M}=\cO[\![\Gamma_{M}]\!]$. 

\subsubsection*{Selmer groups}\label{ss:Sel-I}
%
%Let $g \in S_{2}(\Gamma_{0}(\fn))$ be a Hilbert modular newform and $p$ an odd prime of good ordinary reduction. 
%
%For $g\in S_2(\Gamma_0(\fn))$ as above, 
We consider the $\cO[G_M]$-module 
\[
\mathcal{M}_{g}=T_g\otimes_{\Z_{p}} \Lambda_{M}^{\vee},
\]
where $\Lambda_M^\vee={\rm Hom}_{\rm cts}(\Lambda_M,\Q_p/\Z_p)$ denotes the Pontryagin dual, 
and 
 $G_{M}$ acts on 
$\Lambda_{M}$ and $\Lambda_{M}^{\vee}$ via 
$\Psi: G_{M}\twoheadrightarrow \Gamma_{M} \hookrightarrow \Lambda_{M}^{\times}$ and $\Psi^{-1}$, respectively. For every prime $w$ of $M$ above $p$, 
put $\mathcal{M}_{g,w}^{+}={\rm{Fil}}_{w}^{+}(T_g) \otimes_{\Z_{p}}\Lambda_{M}^{\vee}$. Let $\Sigma$ be a finite set of places of $M$ containing the primes above $\fn p\infty$, let $M^\Sigma$ be the maximal extension of $M$ unramified outside $\Sigma$, and define the \emph{ordinary Selmer group} by  
\[
\rH^{1}_{\CF_{\Lambda}}(M,\mathcal{M}_{g}):=\ker\bigg{\{}\rH^{1}(M^\Sigma/M,\mathcal{M}_{g}) \ra 
\prod_{w \in \Sigma, w\nmid p}\rH^{1}(M_{w},\mathcal{M}_{g}) \times \prod_{w|p}\rH^{1}(I_{w},\mathcal{M}_{g}/\mathcal{M}_{g,w}^{+})\bigg{\}}.
\]
We put $\mathfrak{X}_{\rm ord}(g/M_\infty)=\rH^1_{\CF_{\Lambda}}(M,\mathcal{M}_g)^\vee$ to denote the Pontryagin dual, and for any subextension $M'$ of $M_\infty/M$ let $\mathfrak{X}_{\rm ord}(g/M')$ be the analogously defined Selmer group with ${\rm Gal}(M'/M)$ in place of $\Gamma_M$.

We shall also consider the \emph{Greenberg Selmer group}
\[
\rH^{1}_{\CF_{\rm Gr}}(M,\mathcal{M}_{g}):=\ker\bigg{\{}\rH^{1}(M^\Sigma/M,\mathcal{M}_{g}) \ra 
\prod_{w \in \Sigma, w\nmid p}\rH^{1}(M_{w},\mathcal{M}_{g}) \times\prod_{w\mid\overline{v}}\rH^{1}(I_{w},\mathcal{M}_{g})\bigg{\}}
\]
and its Pontryagin dual $\mathfrak{X}_{\rm Gr}(g/M_\infty)$, as well as their variants for any $M'$ as above.

%For $S=\Sigma\setminus\{v\,\vert\, p\infty\}$, let 
%$$
%H^{1}_{\CF_{\mathrm{ord}}^{S}}(M,M_{g}):=\ker\bigg{\{}H^{1}(G_{\Sigma},M_{g}) \ra 
%\prod_{w|p} H^{1}(I_{w}, T_{g}/T_{g,w}^{+}\otimes_{\Z_{p}} \Lambda_{M}^{\vee})
%\bigg{\}}.
%$$
%be the $S$-relaxed analogue. 
%Let $S(g/M)$ denote $H^{1}_{\CF_{\mathrm{ord}}}(M,M_{g})$ for brevity and $X(g/M)$ be its Pontryagin dual. 
%Analogously, define $X^{S}(g/M)$. Occasionally, we let $S(g/M_\infty)$ and $X(g/M_\infty)$ denote $S(g/M)$ and $X(g/M)$. 
%For a subextension $N$ of $M_\infty/M$, let $S(g/N)$ and $X(g/N)$ be the analogously defined Selmer groups over $N$. 

\subsubsection*{$p$-adic $L$-functions} 
%Let $g \in S_{2}(\Gamma_{0}(\fn))$ be a Hilbert modular newform over a totally real field $F$ and $p$ an odd prime of good ordinary reduction unramified in $F$. 

Assume that the prime $p$ is odd and unramified in $F$. Let 
$M/F$ be a CM quadratic field extension satisfying \eqref{spl-I} and such that 
\begin{equation}\label{exc-I}\tag{$\Delta$}
\text{$M$ is not contained in $H_{F}$, and any prime ramified in $F/\Q$ splits in $M$.}
\end{equation}
Here $H_F$ denotes the Hilbert class field of $F$.
Let 
\[
\CL_p^{}(g/M) \in \Lambda_{M} \otimes\Q_p
\]
be the associated $(d+1)$-variable $p$-adic $L$-function as in \cite[\S7.3]{wan-hilbert}, 
%For a subextension $N$ of $M_\infty/M$, let $\CL(g/N)$ denote the image of $\CL(g/M)$ under the projection  $\Lambda_{M} \otimes_{\Z_{p}} \Q_{p} \twoheadrightarrow \cO[\![\Gal(N/M)]\!]\otimes_{\Z_{p}}\Q_{p}$. 
%
which interpolates the central $L$-values $L^{\rm alg}(g/M \otimes \chi, 1)$ 
as $\chi$ varies over finite order characters of $\Gamma_M$ (cf.~\cite[Thm.~82(i)]{wan-hilbert}).
If an underlying Hecke algebra is Gorenstein, then \cite[Thm.~82(ii)]{wan-hilbert}  shows the inclusion $\CL_p(g/M) \in \Lambda_{M}$.

\subsection{Iwasawa Main Conjecture} %in $(d+1)$ variables}

\begin{conj}\label{g-var-IMC}
Let $g \in S_{2}(\Gamma_{0}(\fn))$ be a Hilbert modular newform  over a totally real field $F$ and $p$ an odd prime unramified in $F$ and good ordinary for $g$. 
Let $M/F$ be a CM quadratic extension satisfying \eqref{spl-I}
and such that
\begin{equation}\label{irr_M}\tag{irr$_M$}
\text{$\bar{\rho}_{g}$ is irreducible as $G_M$-representation.}
\end{equation}
Then $\mathfrak{X}_{\rm ord}(g/M_\infty)$ is $\Lambda_M$-torsion, with
\[
{\rm ch}_{\Lambda_M}\bigl(\mathfrak{X}_{\rm ord}(g/M_\infty)\bigr)=\bigl(\mathcal{L}_p(g/M)\bigr).
\]
\end{conj} 

\begin{rem} 
Without conditions \eqref{ur}, and \eqref{irr_M}, the conjecture is still expected to hold, with the equality of characteristic ideals  being possibly up to tensoring with $\Q_p$.
\end{rem}

\section{Main Conjectures over quartic CM fields} 

We describe a consequence of a result \cite{wan-hilbert} towards Conjecture \ref{g-var-IMC} which will be central to the proofs of main results. 

%Let $g \in S_{2}(\Gamma_{0}(N))$ be an elliptic newform and $F_g$ the Hecke field. 
%Let $p$ be a prime of good ordinary reduction 
%with $\lambda$ a prime of the Hecke field $F_g$ over $p$ satisfying 
%\begin{equation}\label{ord}\tag{ord}
%a(p,g)\in \cO_{F_{g,\lambda}}^\times, 
%\end{equation}
%and $\rho: G_{\Q}\rightarrow{\rm GL}_{2}(\cO_{F_{g,\lambda}})$ the associated Galois representation. 
%Let $F$ be a real quadratic field and $g_{F}$ the base change of $g$ to $F$ (cf.~\cite{DN},~\cite{AC}). The corresponding Galois representation is given by $\rho_{F}:=\rho|_{G_{F}}$. 

\subsection{A hypothesis} The results of \cite{wan-hilbert} are conditional on the following hypothesis. 
\begin{hyp}\label{H}
%Consider the following: 
\noindent\
\begin{itemize}
\item[(H1)] $\bar{\rho}_g|_{G_{F(\zeta_{p})}}$ is absolutely irreducible, and for $p=5$ the following case is excluded: the projective image $\bar{G}$ of $\bar{\rho}_g\vert_{G_F}$ is isomorphic to ${\rm PGL}_{2}(\mathbb{F}_p)$ and the mod $p$ cyclotomic character factors through $G_{F}\ra \bar{G}^{\rm ab}\simeq \Z/2\Z$ (in particular $[F(\zeta_{5}):F]=2$).
\item[(H2)] There is a minimal modular lifting of $\bar{\rho}_g\vert_{G_F}$ (cf.~\cite[Def.~6.11]{Fu}).
\item[(H3)] For any finite place $v$ of $F$, if $\bar{\rho}_{g}\vert_{G_{F_{v}}}$ is absolutely irreducible and $\bar{\rho}_{g}\vert_{I_{{v}}}$ is absolutely reducible, then $q_{v}\not\equiv -1 \pmod{p}$. 
\end{itemize}
\noindent Here $I_{v}\subset G_{F_v}$ are inertia and decomposition groups at $v$, respectively, and $q_v$ denotes the size of the residue field of $v$.
\end{hyp}

\begin{rem}\label{rem-H}\noindent\
\begin{itemize}
\item[(i)] The hypothesis (H1) implies that a certain Hecke algebra is Gorenstein (cf.~\cite[Thm.~11.1]{Fu}), and so  $\CL_p(g/M)\in \Lambda_M$. (Note that here $\bar{\rho}_{g}$ is automatically $p$-distinguished in the sense of \emph{loc.\,cit.})
\item[(ii)] Under (H3), the exceptional case $0_E$ in \cite[p.~16]{Fu} does not occur, and hence the results of \cite{Fu} apply to the setting of \cite[\S7--9]{wan-hilbert} (cf.~\cite[p.~57]{Fu}).
\item[(iii)] The case excluded by (H1) does not occur for $g$ corresponding to elliptic curves (cf.~\cite[Prop.~9.8]{Fu}). 
\end{itemize}
\end{rem} 
\subsection{An Eisenstein congruence divisibility}
\begin{thm}[Wan]\label{3-var-IMC}
Let $g\in S_{2}(\Gamma_{0}(N))$ be an elliptic newform, and let $p>3$ be a prime of good ordinary reduction for $g$. 
Let $F$ be a real quadratic field with $(pN,D_F)=1$, and  let $g_F$ denote the base-change of $g$ to $F$. 
Let $M/F$ be a CM quadratic extension satisfying \eqref{spl-I}, \eqref{exc-I}, and $(N\cO_{F},D_{M/F})=(1)$. Write 
\[
N\cO_F = \fn^{+}\fn^{-},
\]
with $\fn^{+}$ (resp. $\fn^-$) divisible only by primes which are split (resp. inert) in $M/F$. Suppose that: 
%For the set-up as above: 
\begin{itemize}
\item[(i)] $\bar{\rho}_{g_F}=\bar{\rho}_g\vert_{G_F}$ satisfies \eqref{irr_M}. 
\item[(ii)] Hypothesis \ref{H} holds. 
\item[(iii)] $\fn^-$ is the square-free product of an even number of primes. 
\item[(iv)] $\bar{\rho}_{g_F}$ is ramified at every prime dividing $\fn^-$.
\end{itemize}
Then we have the divisibility
\[
\bigl(\CL_p(g_F/M)\bigr) \supset {\rm ch}_{\Lambda_M}\bigl(\mathfrak{X}_{\rm ord}(g/M_\infty)\bigr)
\]
in $\Lambda_M$.
\end{thm}

\begin{proof} 
By \cite[Thm.~3]{wan-hilbert}, we have the divisibility 
\[
\bigl(\CL_p({\bf{g}}_{F}/M)\bigr) \supset {\rm ch}_{\BI[\![\Gamma_{M}]\!]}\bigl(\mathfrak{X}_{\rm ord}({\bf{g}}/M_\infty)\bigr)
\]
in $\BI[\![\Gamma_{M}]\!]$, where ${\bf{g}}_{F}$ denotes the parallel weight Hida family passing through $g_F$ 
(cf.~\cite{Hi1,Hi2}) and $\BI$ is its coefficient ring. 
%The specialization map $\BI \ra \cO$ corresponding to the $p$-ordinary stabilisation of $g_F$ gives rise to a divisibility 
%Note that $X(g_{F}/M)$ is torsion by Proposition \ref{Ka-tor} (i) (cf.~\cite[~\S3.5.1]{SU}) and  so 

Since $\mathfrak{X}_{\rm ord}({\bf{g}}/M_\infty)$ specializes to $\mathfrak{X}_{\rm ord}(g/M_\infty)$ under the map induced by the specialization $\BI \ra \cO$ corresponding to the $p$-ordinary stabilization of $g_F$ (cf.~\cite[(3.5)]{skinner-urban}), and 
$\CL_p(g_{F}/M)$ is defined as an analogous specialization of $\CL_p({\bf{g}}_{F}/M)$, the assertion follows. 
\end{proof}
\begin{rem}
As indicated in Remark~\ref{rem:cyc}(ii), 
Wan uses the above divisibility in conjunction with base change to remove the condition \eqref{eq:mult} from the result of \cite{skinner-urban} up to tensoring with $\Q_p$ (cf.~\cite[Thm.~4]{wan-hilbert}).
\end{rem}

\section{Main Conjectures over imaginary quadratic fields}

We collect some known results in the direction of Conjecture~\ref{g-var-IMC} (and some variants) in the case $F=\Q$. Although some of these results are known under weaker hypotheses, here we shall assume that
\begin{equation}\label{irr_K}\tag{irr$_K$}
\text{$\bar{\rho}_{g}$ is irreducible as $G_K$-representation,}
\end{equation}
where $K=M$ is imaginary quadratic in this section.
%as this will suffice for our applications. %in this appendix under \eqref{eq:irr}.

We refer the reader to $\S\S{1.2}$ and ${1.4}$ of \cite{eisenstein_cyc} for a review of the construction and interpolation property of the two-variable $p$-adic $L$-functions $\mathcal{L}_p^{\rm PR}(g/K)\in\Lambda_K$ and $\mathcal{L}_p^{\rm Gr}(g/K)\in\Lambda_K^{\rm ur}$ appearing below.

\subsection{Two-variable Main Conjectures}

\begin{conj}\label{conj:2-var-ord-IMC}
Let $g \in S_{2}(\Gamma_{0}(N))$ be a newform and $p>2$ a  prime of good ordinary reduction for $g$. 
Let $K$ be an imaginary quadratic field satisfying 
%\eqref{eq:intro-spl}, $(D_K,N)=1$ and 
\eqref{irr_K}. Then $\mathfrak{X}_{\rm ord}(g/K_\infty)$ is $\Lambda_K$-torsion, with
\[
\bigl({\CL}_p^{\rm PR}(g/K))={\rm ch}_{\Lambda_K}\bigl(\mathfrak{X}_{\rm ord}(g/K_\infty)\bigr).
\]
\end{conj} 

Note that it follows from the comparison of $p$-adic $L$-functions in \cite[Prop.~84]{wan-hilbert} that Conjecture~\ref{conj:2-var-ord-IMC} is nothing but  Conjecture~\ref{g-var-IMC} in the current setting. 
%Conjecture~\ref{3-var-IMC} in the case $F=\Q$.

\begin{conj}\label{conj:2-var-Gr-IMC}
Let $g \in S_{2}(\Gamma_{0}(N))$ be a newform and $p>2$ a prime of good reduction for $g$.  
Let $K$ be an imaginary quadratic field satisfying \eqref{eq:intro-spl}. %and $(D_K,N)=1$. 
%For the set-up as above: 
Then $\mathfrak{X}_{\rm Gr}(g/K_\infty)$ is $\Lambda_K$-torsion, with 
\[
\bigl({\CL}_p^{\rm Gr}(g/K)\bigr)={\rm ch}_{\Lambda_K}(\mathfrak{X}_{\rm Gr}(g/K_\infty)\bigr)
\]
as ideals in $\Lambda_K^{\rm ur}$.
\end{conj} 

%Underlying the following key result from \cite{BSTW} is a refinement of the Beilinson--Flach classes of \cite{explicit} and their explicit reciprocity laws, allowing us to pass between the two preceding two-variable main conjectures.

The central result of \cite{BSTW} is the existence of two-variable zeta element for $E$ over $K$,  
%(see also~\cite{explicit}), 
which leads to the following useful equivalence between the preceding two-variable main conjectures of different nature. 

\begin{thm}\label{Eq-MC}
Let $g \in S_{2}(\Gamma_{0}(N))$ be an elliptic newform, 
and $p \nmid 2N$ an ordinary prime for $g$. Let $K$ be an imaginary quadratic field satisfying \eqref{eq:intro-spl}, $(D_{K},N)=1$, and \eqref{irr_K}. Then the following are equivalent:
\begin{itemize}
\item[(i)] $\mathfrak{X}_{\rm ord}(g/K_\infty)$ is $\Lambda_K$-torsion, with
\[
\bigl({\CL}_p^{\rm PR}(g/K))\supset{\rm ch}_{\Lambda_K}\bigl(\mathfrak{X}_{\rm ord}(g/K_\infty)\bigr)\quad\textrm{in $\Lambda_K\otimes\Q_p$.
}
\]
%in $\Lambda_K\otimes\Q_p$.
\item[(ii)] $\mathfrak{X}_{\rm Gr}(g/K_\infty)$ is $\Lambda_K$-torsion, with
\[
\bigl({\CL}_p^{\rm Gr}(g/K)\bigr)\supset{\rm ch}_{\Lambda_K}(\mathfrak{X}_{\rm Gr}(g/K_\infty)\bigr)\quad\textrm{in $\Lambda_K^{\rm ur}\otimes\Q_p$.}
\]
%in $\Lambda_K^{\rm ur}\otimes\Q_p$.
\end{itemize}
The same conclusion holds for the opposite divisibilities, and before inverting $p$. In particular, Conjecture~\ref{conj:2-var-ord-IMC} and Conjecture~\ref{conj:2-var-Gr-IMC} are equivalent. 
\end{thm}

\begin{proof}
This is shown in \cite[\S9.3.2]{BSTW} (cf. \cite[Prop.~3.2.1]{eisenstein_cyc} or \cite[\S{3.3}]{ICTS}) building on a pair of four-term exact sequences coming from Poitou--Tate duality. 
%Note that an ingredient in the proof is the nonvanishing of the $p$-adc $L$-functions $\CL_p^{\rm PR}(g/K)$ and $\CL_v(g/K)$.
\end{proof}

Taking the direct sum of two pairs of four-term exact sequences as in the proof of Theorem~\ref{Eq-MC},  one deduces the following.

\begin{cor}\label{cor:2-var-equiv}
Let $g\in S_2(\Gamma_0(N))$ and $g'\in S_2(\Gamma_0(N'))$ be newforms, $p\nmid 2NN'$ an ordinary prime for both $g$ and $g'$, and $K$ an imaginary quadratic field satisfying \eqref{eq:intro-spl} and $(D_K,NN')=1$, and such that \eqref{irr_K} holds for both $g$ and $g'$. Then the following are equivalent:
\begin{itemize}
\item[(i)] $\bigl(\CL_p^{\rm PR}(g/K)\cdot \CL_{p}^{\rm PR}(g'/K)\bigr)\supset
{\rm ch}_{\Lambda_K}(\mathfrak{X}_{\rm ord}(g/K_\infty))\cdot{\rm ch}_{\Lambda_K}(\mathfrak{X}_{\rm ord}(g'/K_\infty))$.
\item[(ii)] $\bigl(\CL_p^{\rm Gr}(g/K)\cdot \CL_{p}^{\rm Gr}(g'/K)\bigr)\supset
{\rm ch}_{\Lambda_K}(\mathfrak{X}_{\rm Gr}(g/K_\infty))\cdot{\rm ch}_{\Lambda_K}(\mathfrak{X}_{\rm Gr}(g'/K_\infty))$.
\end{itemize}
Moreover, the same holds for the opposite divisibility.
\end{cor}

\subsection{Anticyclotomic Main Conjectures}

A refinement of Kolyvagin's methods (in the style of \cite{howard} and \cite{nekovar}) developed in \cite{eisenstein,eisenstein_cyc} yields the first part of the following result. 
\begin{thm}\label{HMC-lm}
Let $g\in S_{2}(\Gamma_{0}(N))$ be an elliptic newform and $p$ an odd prime of good ordinary reduction for $g$. Let $K$ be an imaginary quadratic field satisfying \eqref{eq:intro-disc}, \eqref{eq:intro-Heeg}, and \eqref{irr_K}. Then the following hold:
%For the set-up as above:  
\begin{itemize}
%\item[(a)] The Heegner class $\kappa \in H^{1}_{\CF_{\mathrm{ord}}}(K,T_{g}\otimes_{\BZ_{p}}\Lambda_{K}^{\rm{ac}})$ is not $\Lambda_{K}^{\rm{ac}}$-torsion,
\item[(a)] Both $\mathfrak{X}_{\rm ord}(g/K_\infty^-)$  and $\rH^1_{\Fcal_\Lambda}(K,T_g\otimes\Lambda_K^-)$ have $\Lambda_K^-$-rank one, and 
\[
{\rm ch}_{\Lambda_K^-}\bigl(\X_{\rm ord}(g/K_\infty^-)_{\rm tor}\bigr)\supset{\rm ch}_{\Lambda_K^-}\bigl(\rH^1_{\Fcal_\Lambda}(K,T_g\otimes\Lambda_K^-)/(\boldsymbol{\kappa}_1^{\rm Heeg})\bigr)^2\quad\textrm{in $\Lambda_K^-\otimes\Q_p$}.
\]
%where the subscript ${\rm tor}$ denotes the $\Lambda_K^-$-torsion submodule.
\item[(b)] If $K$ also satisfies \eqref{eq:intro-spl}, then 
$\mathfrak{X}_{\rm Gr}(g/K_\infty^-)$ is $\Lambda_K^-$-torsion, and 
\[
{\rm ch}_{\Lambda_K^-}(\X_{\rm Gr}(g/K_\infty^-))\Lambda_K^{-,{\rm ur}}\supset(\Lcal_p^{\rm BDP}(g/K))\quad\textrm{in $\Lambda_K^{-,{\rm ur}}\otimes\Q_p$.}
\]
\end{itemize}
Moreover, if \eqref{eq:irred} holds, then both divisibilities hold integrally. 
\end{thm}

\begin{proof}
Part (a) is contained in \cite[Thm.~5.5.2]{eisenstein_cyc}, and part (b) then follows from  \cite[Thm.~5.2]{BCK}. Under \eqref{eq:irred}, part (a) follows from \cite[Thm.~B]{howard}, and part (b) again from \cite[Thm.~5.2]{BCK}.
\end{proof}

The following vanishing of $\mu$-invariant result based on Hida's ideas \cite{hidamu} will also play an important role in our arguments.

\begin{prop}\label{mu}
Let $g\in S_2(\Gamma_0(N))$ be an elliptic newform with good reduction at $p>2$, and suppose $K$ is an imaginary quadratic field satisfying  \eqref{eq:intro-disc}, \eqref{eq:intro-Heeg}, \eqref{eq:intro-spl}, and \eqref{irr_K}. Then 
\[
\mu(\CL_{p}^{\rm Gr}(g/K))=\mu(\CL_{p}^{\rm BDP}(g/K))=0.
\]
\end{prop}

\begin{proof}
By \cite[Thm.~B]{hsieh-special}, $\CL_{p}^{\rm BDP}(g/K)$ has vanishing $\mu$-invariant. Since a direct comparison of the interpolation properties shows that the projection of $\CL_p^{\rm Gr}(g/K)$ to $\Lambda_K^{-,{\rm ur}}$ generates the same ideal as $\CL_p^{\rm BDP}(g/K)$ (see \cite[Prop.~1.4.5]{eisenstein_cyc}), the result follows.
\end{proof}

\section{Base change} 

%We can now describe the proof of the main results in \S\ref{subsec:app-results}. 

%\subsubsection{Factorisations} 

Let $g \in S_{2}(\Gamma_{0}(N))$ be an elliptic newform, 
and $p>3$ a good ordinary prime for $g$ such that \eqref{eq:irr} holds.  
\subsection{Base change} Let $M/F$ be a CM quadratic extension of a real quadratic field $F$ for the form $M=FK$ with $K$ an imaginary quadratic field. 

For the $\Z_p^2$-extension $\tilde{K}_\infty := FK_\infty$ of $M$, put $\tilde{\Lambda}_{K}:=\cO[\![\Gal(\tilde{K}_{\infty}/M)]\!]$ and let 
$$
\pi_K:\Lambda_M \rightarrow\tilde{\Lambda}_K\simeq\Lambda_K
$$
denote the map arising from the projection $\Gal(M_\infty/M)\twoheadrightarrow \Gal(\tilde{K}_\infty/M)$. 

\begin{lemma}\label{Sel-fac}
%Let $g \in S_{2}(\Gamma_{0}(N))$ be an elliptic newform, 
%and $p \nmid 2N$ an ordinary prime. %so that ({\rm{irr}}) holds. 
%Let $K$ be an imaginary quadratic field satisfying \eqref{spl} and  $(D_{K},N)=1$. 
%and \eqref{irr_L} with $L=K$. 
%Let $F$ be a real quadratic field. 
%satisfying \eqref{in} and \eqref{disc_F}. 
%Let $M$ denote the compositum $FK$ and suppose \eqref{irr_L} for $L=M$ holds for $L=M$. Let $g^F$ denote the $F$-quadratic twist of $g$ and $g_F$ the base change of $g$ to $F$. Then we have 
We have the divisibility
\[
\pi_K\bigl({\rm ch}_{\Lambda_M}(\mathfrak{X}_{\rm ord}(g/M_\infty))\bigr)\supset{\rm ch}_{\Lambda_K}\bigl(\mathfrak{X}_{\rm ord}(g/K_\infty)\bigr)\cdot{\rm ch}_{\Lambda_K}\bigl(\mathfrak{X}_{\rm ord}(g^{F}/K_\infty)\bigr)
\]
in $\Lambda_K$, where $g^F$ is the twist of $g$ by the quadratic character corresponding to $F/\Q$.
\end{lemma}

\begin{proof}
%All the references are to \cite{skinner-urban}. 
It follows readily from Shapiro's lemma and \cite[Prop.~3.6 and Prop.~3.7]{skinner-urban} that the restriction map $\rH^1(M,T_g\otimes\tilde{\Lambda}_K^\vee)\rightarrow\rH^1(M,\CM_g)[I_K]$, where $I_K=\ker(\pi_K)$, induces $\Lambda_K$-module isomorphisms
\begin{align*}
\rH^1_{\Fcal_\Lambda}(M,\mathcal{M}_g)[I_K]&\simeq\rH^1_{\Fcal_\Lambda}(M,T_g\otimes\tilde{\Lambda}_K^\vee)\\
&\simeq\rH^1_{\Fcal_\Lambda}(M,T_g\otimes\Lambda_K^\vee)\oplus\rH^1_{\Fcal_\Lambda}(M,T_g\otimes\eta_F\otimes\Lambda_K^\vee). 
\end{align*}
Here $\eta_F$ is the non-trivial character of $\Gal(F/\Q)\simeq \Gal(M/K)$ and so $$T_g\otimes\eta_F\simeq T_{g^F}.$$ Since by \cite[Cor.~3.8]{skinner-urban} we have the divisibility
\[
\pi_K\bigl({\rm ch}_{\Lambda_M}(\mathfrak{X}_{\rm ord}(g/M_\infty))\bigr)\supset{\rm ch}_{\Lambda_K}\bigl(\mathfrak{X}_{\rm ord}(g/M_\infty)/I_K\bigr)
\]
in $\Lambda_M/I_K\simeq\Lambda_K$, the proof concludes.
\end{proof}

\begin{lemma}\label{pL-fac}
Let the setting be as in Lemma \ref{Sel-fac}. Then we have 
$$
\bigl(\pi_K(\CL_p(g_{F}/M))\bigr)=(\CL_p^{\rm PR}(g/K)\cdot \CL_p^{\rm PR}(g^{F}/K)\bigr)
$$
in $\Lambda_K \otimes\Q_p$.
\end{lemma}
\begin{proof}
This just follows from a direct comparison of the interpolation properties %of the $p$-adic $L$-functions 
(cf. \cite[Prop.~84]{wan-heegner}). 
\end{proof}

%\begin{remark}
%The conclusions of Lemma \ref{Sel-fac} and \ref{pL-fac} hold without the conditions \eqref{in} and \eqref{disc_F}.  However, the conditions do play a role in the next subsection. 
%\end{remark}

\subsection{Proofs of main results}\label{ss:proofs}

We can now complete the proof of the results stated in \S\ref{subsec:app-results}. We begin with the following key proposition.

\begin{prop}\label{MC-lb-Sel}
Let $g \in S_{2}(\Gamma_{0}(N))$ be an elliptic newform 
and $p\geq 5$ a prime of good ordinary reduction for $g$. Let $K$ be an imaginary quadratic field satisfying  \eqref{eq:intro-disc}, \eqref{eq:intro-Heeg}, \eqref{eq:intro-spl}, and \eqref{irr_K}. Suppose $F$ is a real quadratic field of discriminant $D_F$ satisfying the following hypotheses:
\begin{itemize}
\item[(i)] $p$ is inert in $F$.
\item[(ii)] $D_F$ is odd and every prime dividing $D_F$ splits in $K$.
\item[(iii)] Every prime $\ell\vert N$ is inert in $F$ if $\ell \equiv -1 \pmod {p}$, and is split in $F$ otherwise.
\item[(iv)] $\eqref{irr_M}$ holds for $M=FK$.
\item[(v)] $\bar{\rho}_g|_{G_{F(\zeta_p)}}$ is irreducible.
\item[(vi)] If $p=5$, then $F\neq\Q(\zeta_5)^+$.
\end{itemize}
%
%\begin{equation}\label{inr}\tag{inr}
%\textrm{$p$ is inert in $K$;}
%\end{equation}
%\begin{equation}\label{disc_F}\tag{disc$_F$}
%\textrm{$D_F$ is odd and every primes dividing $D_F$ splits in $K$;}
%\end{equation}
%\begin{equation}\label{TW}\tag{TW}
%\text{$\bar{\rho}|_{G_{F(\zeta_p)}}$ is irreducible;}
%\end{equation}
%\begin{equation}\label{n-exc}\tag{n-exc}
%\text{any prime $\ell$ dividing $N$ is inert in $F$ if $\ell \equiv -1 \pmod {p}$, and is split in $F$ otherwise;}
%\end{equation}
%and such that \eqref{irr_M} holds for $M=FK$ and $F\neq\Q(\zeta_5)^+$ if $p=5$. Let $g^F$ denote the $F$-quadratic twist of $g$. 
%
Then we have the divisibilities 
\[
\bigl(\CL_p^{\rm PR}(g/K)\cdot \CL_{p}^{\rm PR}(g^{F}/K)\bigr)\supset
{\rm ch}_{\Lambda_K}(\mathfrak{X}_{\rm ord}(g/K_\infty))\cdot{\rm ch}_{\Lambda_K}(\mathfrak{X}_{\rm ord}(g^{F}/K_\infty))\quad\textrm{in $\Lambda_{K}$,}
\]
%in $\Lambda_{K}$, and 
and
\[
\bigl(\CL_p^{\rm Gr}(g/K)\cdot \CL_{p}^{\rm Gr}(g^{F}/K)\bigr)\supset
{\rm ch}_{\Lambda_K}(\mathfrak{X}_{\rm Gr}(g/K_\infty))\cdot{\rm ch}_{\Lambda_K}(\mathfrak{X}_{\rm Gr}(g^{F}/K_\infty))\quad\textrm{in $\Lambda_K^{\rm ur}$.}
\]
%in $\Lambda_K^{\rm ur}$.
\end{prop}

\begin{proof}
Note that the CM quadratic extension $M/F$ and the residual representation  $\bar{\rho}_{g}|_{G_{F}}$ satisfy the hypotheses of Theorem~\ref{3-var-IMC}. Indeed, the conditions $(pN,D_F)=1$ and $(N\mathcal{O}_F,D_{M/F})=1$ are clear, and hypotheses (i)--(iii)  and \eqref{eq:intro-Heeg} imply \eqref{spl-I} and \eqref{exc-I}. Further 
(ii) implies $\fn^{-}=\cO_F$ and so hypotheses (iii) and (iv) of Theorem \ref{3-var-IMC} are vacuous. Hypothesis~\ref{H} is readily verified: (H1) follows by (v) and (vi), (H2) follows as in \cite[Thm.~103]{wan-heegner} (indeed, by our choice of $F$, the base change to $F$ of a minimal  modular lifting of $\bar{\rho}_g$ gives a minimal modular lifting of $\bar{\rho}_{g_F}$), and (H3) also follows by condition (iii) on $F$.

Now Theorem \ref{3-var-IMC} leads to the divisibility 
\[
\bigl(\pi_K(\CL_p(g_F/M))\bigr) \supset \pi_K\bigl({\rm ch}_{\Lambda_M}\bigl(\mathfrak{X}_{\rm ord}(g/M_\infty))\bigr)
\]
in $\Lambda_K$, and so by Lemmas~\ref{Sel-fac} and \ref{pL-fac}, we have 
\[
\bigl(\CL_p^{\rm PR}(g/K)\cdot \CL_p^{\rm PR}(g^{F}/K)\bigr)\supset
{\rm ch}_{\Lambda_K}(\mathfrak{X}_{\rm ord}(g/K))\cdot{\rm ch}_{\Lambda_K}(\mathfrak{X}_{\rm ord}(g^{F}/K))
\]
in $\Lambda_{K}\otimes\Q_{p}$. In turn Proposition \ref{Eq-MC} (and its proof) implies 
\begin{equation}\label{div-rat}
\bigl(\CL_p^{\rm Gr}(g/K)\cdot \CL_p^{\rm Gr}(g^{F}/K)\bigr)\supset
{\rm ch}_{\Lambda_K}(\mathfrak{X}_{\rm Gr}(g/K))\cdot{\rm ch}_{\Lambda_K}(\mathfrak{X}_{\rm Gr}(g^{F}/K))
\end{equation}
in $\Lambda_K^{\rm ur}\otimes_{}\Q_p$. Since 
$\mu(\CL_p^{\rm Gr}(g/K)\cdot \CL_p^{\rm Gr}(g^{F}/K))=0$ by Proposition~\ref{mu}, the divisibility \eqref{div-rat} holds integrally in $\Lambda_K^{\rm ur}$. 
By again appealing to Proposition \ref{Eq-MC}, the proof concludes. 
\end{proof}

\begin{rem} 
In the case $p=5$, if $g$ corresponds to an elliptic curve $E/\Q$, then the condition $F\neq \Q(\zeta_5)^{+}$ is inessential (cf.~Remark~\ref{rem-H}(iii)).  
\end{rem}

The existence of $F$ satisfying the conditions in Proposition~\ref{MC-lb-Sel} is easily verified:
  
\begin{lemma}\label{exs}
Let $g \in S_{2}(\Gamma_{0}(N))$ be an elliptic newform 
and $p\geq 5$ a prime of good ordinary reduction for $g$ such that $\bar{\rho}_g$ is irreducible as $G_\Q$-representation.  
%Let $(g,p,K)$ be as in Proposition~\ref{MC-lb-Sel}. 
Then there exist  an imaginary quadratic field $K$ satisfying  \eqref{eq:intro-disc}, \eqref{eq:intro-Heeg}, \eqref{eq:intro-spl}, and \eqref{irr_K}, and a real quadratic field $F$ satisfying ${\rm (i)}$--${\rm (vi)}$  in Proposition~\ref{MC-lb-Sel}.
\end{lemma}

\begin{proof}
Since $p\nmid N$ by hypothesis, (i) and (iii) are independent splitting conditions  which hold for a positive proportion of real quadratic fields $F$; fix one such $F$ with odd discriminant $D_F$. Similarly, \eqref{eq:intro-spl}, \eqref{eq:intro-Heeg}, and (ii) hold for a positive proportion of imaginary quadratic fields $K$, and we can fix one such $K$ satisfying these conditions in addition to \eqref{eq:intro-disc} and \eqref{irr_K}.
%
%Note that (i), (ii), and (iii) are independent splitting conditions which hold for a positive proportion of real quadratic fields $F$. Indeed, this is a simple consequence of Chebotarev density theorem. 

In view of \eqref{irr_K}, if $FK$ is not a subfield of the splitting field of $\rho_g\vert_{G_K}$, then $\eqref{irr_M}$ holds for $M=FK$. For such $F$, if $\bar{\rho}_g|_{G_{F(\zeta_{p})}}$ is reducible, then $\bar{\rho}_{g}\vert_{G_F}$ is induced by an index $2$ subgroup $G_L$ of $G_F$ which contains $G_{F(\zeta_p)}$; but this forces $p=5$ and $F=\Q(\zeta_5)^+$, so the result follows.
\end{proof}
    
%\vskip2mm
\begin{proof}[Proof of Theorem~\ref{thm:cyc}]
Pick an imaginary quadratic field $K$ and a real quadratic field $F$ as in Lemma~\ref{exs}. 

Then by Proposition~\ref{MC-lb-Sel}, 
\begin{equation}\label{lb-2}
\bigl(\CL_p^{\rm PR}(g/K)\cdot \CL_{p}^{\rm PR}(g^{F}/K)\bigr)\supset
{\rm ch}_{\Lambda_K}(\mathfrak{X}_{\rm ord}(g/K_\infty))\cdot{\rm ch}_{\Lambda_K}(\mathfrak{X}_{\rm ord}(g^{F}/K_\infty)).
\end{equation}
By \cite[Prop.~1.2.4]{eisenstein_cyc} and Propositions~3.6 and 3.9 in \cite{skinner-urban}, taking the image under the maps induced by the projection $\pi_+:\Gamma_K\twoheadrightarrow\Gamma_K^+$, from \eqref{lb-2} we get the divisibilities
\begin{equation}\label{eq:cyc-descend}
\begin{aligned}
\bigl(\CL_p^{}(g/\Q)&\cdot\CL_p^{}(g^K/\Q)\cdot\CL_p^{}(g^F/\Q)\cdot\CL_p^{}(g^{FK}/\Q)\bigr)\\
&\supset\pi_+\bigl({\rm ch}_{\Lambda_K}(\mathfrak{X}_{\rm ord}(g/K_\infty))\cdot{\rm ch}_{\Lambda_K}(\mathfrak{X}_{\rm ord}(g^F/K_\infty))\bigr)\\
&\supset{\rm ch}_{\Lambda}(\mathfrak{X}_{\rm ord}(g/\Q_\infty))\cdot{\rm ch}_{\Lambda}(\mathfrak{X}_{\rm ord}(g^K/\Q_\infty))\cdot{\rm ch}_{\Lambda}(\mathfrak{X}_{\rm ord}(g^F/\Q_\infty))\cdot{\rm ch}_{\Lambda}(\mathfrak{X}_{\rm ord}(g^{FK}/\Q_\infty))
\end{aligned}
\end{equation}
in $\Lambda_K^+\simeq\Lambda$. 

On the other hand, by \cite[Thm.~17.4]{kato-euler-systems} we have the divisibility
\begin{equation}\label{eq:kato-div}
\bigl(\CL_p^{}(g/\Q)\bigr)\subset{\rm ch}_{\Lambda}(\mathfrak{X}_{\rm ord}(g/\Q_\infty))
\end{equation}
in $\Lambda$ if hypothesis \eqref{im} holds, and in $\Lambda \otimes \Q_p$ otherwise. 
Similar divisibilities hold for the twists $g^K$, $g^F$, and $g^{FK}$. Noting that a proper divisibility in \eqref{eq:kato-div} would contradict \eqref{eq:cyc-descend}, the proof concludes.
\end{proof}

\begin{proof}[Proof of Theorem~\ref{thm:HPMC} and Theorem~\ref{thm:BDP}] 
Pick a real quadratic field $F$ as in Lemma~\ref{exs}. 

Then by Proposition~\ref{MC-lb-Sel}, 
\begin{equation}\label{lb-1}
\bigl(\CL_p^{\rm Gr}(g/K)\cdot \CL_{p}^{\rm Gr}(g^{F}/K)\bigr)\supset
{\rm ch}_{\Lambda_K}(\mathfrak{X}_{\rm Gr}(g/K_\infty))\cdot{\rm ch}_{\Lambda_K}(\mathfrak{X}_{\rm Gr}(g^{F}/K_\infty)).
\end{equation}
By \cite[Prop.~1.4.5]{eisenstein_cyc} and \cite[Cor.~3.4.2]{jsw}, taking the image under the maps induced by the projection $\pi_-:\Gamma_K\twoheadrightarrow\Gamma_K^-$, from \eqref{lb-1} we get the divisibilities
\begin{align*}
\bigl(\CL_{p}^{\rm BDP}(g/K)\cdot \CL_{p}^{\rm BDP}(g^{F}/K)\bigr)&\supset
\pi_-\bigl({\rm ch}_{\Lambda_K}(\mathfrak{X}_{\rm Gr}(g/K_\infty))\cdot{\rm ch}_{\Lambda_K^-}(\mathfrak{X}_{\rm Gr}(g^{F}/K_\infty))\bigr)\\
&\supset{\rm ch}_{\Lambda_K^-}(\mathfrak{X}_{\rm Gr}(g/K_\infty^-))\cdot{\rm ch}_{\Lambda_K^-}(\mathfrak{X}_{\rm Gr}(g^{F}/K_\infty^-))
\end{align*}
in $\Lambda_K^{-,{\rm ur}}$. 
Together with Theorem~\ref{HMC-lm}, this concludes the proof.
\end{proof}

\begin{proof}[Proof of Theorem~\ref{thm:2var-IMC}] 
Pick an imaginary quadratic field $K$ and two different real quadratic fields $F, F'$ as in Lemma~\ref{exs}. 

For each of the pairs $(K,F)$, $(K,F')$, the equalities of characteristic ideals of Theorem~\ref{thm:BDP}, the divisibility \eqref{lb-1}, and the non-vanishing of the $p$-adic $L$-functions $\mathcal{L}_p^{\rm BDP}(g^\cdot/K)$ for $\cdot\in\{\emptyset,F,F'\}$ yields (by an application of \cite[Lem.~3.2]{skinner-urban}) the equalities
\begin{align*}
\bigl(\CL_p^{\rm Gr}(g/K)\cdot \CL_{p}^{\rm Gr}(g^{F}/K)\bigr)&=
{\rm ch}_{\Lambda_K}(\mathfrak{X}_{\rm Gr}(g/K_\infty))\cdot{\rm ch}_{\Lambda_K}(\mathfrak{X}_{\rm Gr}(g^{F}/K_\infty)),\\
\bigl(\CL_p^{\rm Gr}(g/K)\cdot \CL_{p}^{\rm Gr}(g^{F'}/K)\bigr)&=
{\rm ch}_{\Lambda_K}(\mathfrak{X}_{\rm Gr}(g/K_\infty))\cdot{\rm ch}_{\Lambda_K}(\mathfrak{X}_{\rm Gr}(g^{F'}/K_\infty)),
\end{align*}
and therefore
\begin{equation}\label{eq:product}
\bigl(\CL_p^{\rm Gr}(g/K)^2\cdot \CL_{p}^{\rm Gr}(g^{F}/K)\cdot\CL_{p}^{\rm Gr}(g^{F'}/K)\bigr)=
{\rm ch}_{\Lambda_K}(\mathfrak{X}_{\rm Gr}(g/K_\infty))^2\cdot{\rm ch}_{\Lambda_K}(\mathfrak{X}_{\rm Gr}(g^{F}/K_\infty))\cdot{\rm ch}_{\Lambda_K}(\mathfrak{X}_{\rm Gr}(g^{F'}/K_\infty)).\nonumber
\end{equation}

When $\mathcal{V}=\emptyset$, condition (H3) in Hypothesis~\ref{H} is vacuous, and therefore the argument in the proof of Theorem~\ref{thm:BDP} applies for any real quadratic $F_0$ satisfying conditions (i)-(ii) and (iv)-(vi) in Lemma~\ref{exs}, but not necessarily (iii). Thus taking $F_0$ to be the third real quadratic field in the compositum $FF'$, as above we obtain the equality
\[
\bigl(\CL_p^{\rm Gr}(g^F/K)\cdot \CL_{p}^{\rm Gr}(g^{F'}/K)\bigr)=
{\rm ch}_{\Lambda_K}(\mathfrak{X}_{\rm Gr}(g^F/K_\infty))\cdot{\rm ch}_{\Lambda_K}(\mathfrak{X}_{\rm Gr}(g^{F'}/K_\infty)).
\]
The combination of the last two equalities immediately gives $(\CL_p^{\rm Gr}(g/K))={\rm ch}_{\Lambda_K}(\mathfrak{X}_{\rm Gr}(g/K_\infty))$, which together with Proposition~\ref{Eq-MC} yields the result.
\end{proof}

\bibliography{references}
\bibliographystyle{alpha}
\end{document}